\documentclass{amsart}
\usepackage[utf8]{inputenc}
\usepackage{commath}
\usepackage[ maxnames=4,maxalphanames=4,style=alphabetic-verb,doi=false,
  url=false,
  isbn=false,
  eprint=false
]{biblatex}
\bibliography{main.bib}
\usepackage[margin=1in]{geometry}
\usepackage{graphicx}
\usepackage{caption}
\usepackage{subcaption}
\usepackage{amsmath,amsfonts,amssymb,amsthm,amsaddr,etoolbox}
\usepackage{latexsym,hyperref}
\usepackage[most]{tcolorbox}
\usepackage{xcolor}
\usepackage{scrextend}
\usepackage{quiver}
\usepackage{enumitem}
\usepackage[normalem]{ulem}
\usepackage{mathtools}
\usepackage[linewidth=1pt]{mdframed}
\usepackage{tikz-cd}
\usepackage{float}

\author[R. Easwar, A. Kuber and M. Mittal]{Rohun Easwar, Amit Kuber and Mihir Mittal}
\address{Department of Mathematics and Statistics\\Indian Institute of Technology, Kanpur\\ Uttar Pradesh 208016, India}
\email{rohuneaswar@gmail.com, askuber@iitk.ac.in, mihirmittal24@gmail.com}
\title{On Zeta functions and $\mu$-series of string algebras}
\keywords{string algebra, growth of tame algebras, zeta function, state graph}
\subjclass[2020]{16G60, 11M14}

\newcommand\restr[2]{{
  \left.\kern-\nulldelimiterspace 
  #1 
  \littletaller 
  \right|_{#2} 
  }}

\theoremstyle{plain}
\newtheorem{defn}{Definition}[section]

\newtheorem{lem}[defn]{Lemma}
\newtheorem{conj}{Conjecture}

\newtheorem{prop}[defn]{Proposition}
\newtheorem*{prop*}{Proposition}
\newtheorem*{thm*}{Theorem}
\newtheorem{thm}[defn]{Theorem}
\newtheorem{cor}[defn]{Corollary}

\newtheorem*{claim*}{Claim}

\theoremstyle{remark}
\newtheorem{rem}[defn]{Remark}
\theoremstyle{remark}

\theoremstyle{remark}

\theoremstyle{remark}
\newtheorem{exmp}[defn]{Example}
\theoremstyle{remark}

\theoremstyle{remark}

\theoremstyle{remark}
\newtheorem{fact}[defn]{Fact}
\theoremstyle{remark}

\theoremstyle{remark}

\numberwithin{equation}{section}

\theoremstyle{definition}
\newtheorem{introthm}{Theorem}
\newtheorem{introcor}[introthm]{Corollary}

\makeatletter
\let\save@mathaccent\mathaccent
\newcommand*\if@single[3]{%
  \setbox0\hbox{${\mathaccent"0362{#1}}^H$}%
  \setbox2\hbox{${\mathaccent"0362{\kern0pt#1}}^H$}%
  \ifdim\ht0=\ht2 #3\else #2\fi
  }
\newcommand*\rel@kern[1]{\kern#1\dimexpr\macc@kerna}
\newcommand*\widebar[1]{\@ifnextchar^{{\wide@bar{#1}{0}}}{\wide@bar{#1}{1}}}
\newcommand*\wide@bar[2]{\if@single{#1}{\wide@bar@{#1}{#2}{1}}{\wide@bar@{#1}{#2}{2}}}
\newcommand*\wide@bar@[3]{%
  \begingroup
  \def\mathaccent##1##2{%
    \let\mathaccent\save@mathaccent
    \if#32 \let\macc@nucleus\first@char \fi
    \setbox\z@\hbox{$\macc@style{\macc@nucleus}_{}$}%
    \setbox\tw@\hbox{$\macc@style{\macc@nucleus}{}_{}$}%
    \dimen@\wd\tw@
    \advance\dimen@-\wd\z@
    \divide\dimen@ 3
    \@tempdima\wd\tw@
    \advance\@tempdima-\scriptspace
    \divide\@tempdima 10
    \advance\dimen@-\@tempdima
    \ifdim\dimen@>\z@ \dimen@0pt\fi
    \rel@kern{0.6}\kern-\dimen@
    \if#31
      \overline{\rel@kern{-0.6}\kern\dimen@\macc@nucleus\rel@kern{0.4}\kern\dimen@}%
      \advance\dimen@0.4\dimexpr\macc@kerna
      \let\final@kern#2%
      \ifdim\dimen@<\z@ \let\final@kern1\fi
      \if\final@kern1 \kern-\dimen@\fi
    \else
      \overline{\rel@kern{-0.6}\kern\dimen@#1}%
    \fi
  }%
  \macc@depth\@ne
  \let\math@bgroup\@empty \let\math@egroup\macc@set@skewchar
  \mathsurround\z@ \frozen@everymath{\mathgroup\macc@group\relax}%
  \macc@set@skewchar\relax
  \let\mathaccentV\macc@nested@a
  \if#31
    \macc@nested@a\relax111{#1}%
  \else
    \def\gobble@till@marker##1\endmarker{}%
    \futurelet\first@char\gobble@till@marker#1\endmarker
    \ifcat\noexpand\first@char A\else
      \def\first@char{}%
    \fi
    \macc@nested@a\relax111{\first@char}%
  \fi
  \endgroup
}
\makeatother

\def\R{\mathbb R}

\def\d{\delta}

\def\K{\mathcal K}

\def\B{\mathfrak{b}}
\def\U{\mathfrak{u}}
\def\V{\mathfrak{v}}
\def\W{\mathfrak{w}}

\def\Ba{\mathrm{Ba}}

\begin{document}
\begin{abstract}
Let $\overline\mu_\Lambda(t):=\sum\limits_{m\geq1}\mu_\Lambda(m)t^m$ be the \emph{$\mu$-series} of a finite-dimensional tame algebra $\Lambda$ over an algebraically closed field, where $\mu_\Lambda(m)$ denotes the  minimal number of one-parameter families of $\Lambda$-modules with total dimension $m$. When $\Lambda$ is a string algebra with $\mathrm{Ba}(\Lambda)$ as its set of bands up to cyclic permutation, define the \emph{zeta function} $\zeta_\Lambda(t):=\prod\limits_{\mathfrak b\in\mathrm{Ba}(\Lambda)}(1-t^{|\mathfrak b|})^{-1}$, where $|\B|$ is the length of $\B$. We prove an analogue of the prime number theorem for string algebras and use it to conclude that non-domestic string algebras are of exponential growth. Finally, we show that a string algebra is domestic if and only if its $\mu$-series is rational.
\end{abstract}
\maketitle
\section{Introduction}
In the representation theory of finite-dimensional associative algebras over an algebraically closed field $\K$, the trichotomy result due to Drozd \cite{DrozdRT3} and Crawley-Boevey \cite{CB} partitions all such algebras into finite, tame, or wild representation type--the names signify the tractability of the classification problem for their module categories. 

Recall that a finite-dimensional $\K$-algebra $\Lambda$ of infinite representation type is said to be \emph{tame} if, for each $m\ge1$, there exist finitely many $\Lambda$-$\K[x]$-bimodules $M_1, \cdots M_s$, that are free of finite rank as right $\K[x]$-modules, such that (up to isomorphism) all but finitely many indecomposable $m$-dimensional $\Lambda$-modules are isomorphic to a module of the form $M_i \otimes_{\K[x]} S$ with $S$ a simple $\K[x]$-module. Let $\mu_\Lambda(m)$ denote the minimal number of such bimodules with total dimension $m$.

The \emph{$\mu$-series} of $\Lambda$ is the formal power series $\bar{\mu}_\Lambda(t):=\sum\limits_{m\ge 1}\mu_\Lambda(m)t^m$.

The class of tame algebras can be further partitioned. A tame algebra $\Lambda$ is said to be \begin{itemize}
        \item \emph{domestic} if there exists some $n \geq 0$ such that $\mu_\Lambda(m) \leq n$ for each $m\ge1$;
        \item of \emph{polynomial growth} if there exists some $n \geq 1$ such that $\mu_\Lambda(m) \leq m^n$ for each $m\ge1$;
        \item of \emph{exponential growth} if for each $n \geq 1$ there exists some $m \geq 1$ with $\mu_\Lambda(m) > m^n$.
    \end{itemize}

Among tame algebras, string algebras form an important test subclass for various conjectures, since their module categories are completely understood via combinatorics. Indeed, a classification result of Butler and Ringel \cite{Butler1987AuslanderreitenSW} states that any finite-dimensional module over a string algebra $\Lambda=\K Q/I$--presented using a bound quiver $(Q,I)$--can be described using combinatorial entities called \emph{strings} and \emph{bands}, that are certain walks and certain cyclic walks respectively in the quiver $Q$. As a consequence, the $\mu$-series of a string algebra is essentially determined by the set $\mathrm{Ba}(\Lambda)$ of all bands for $\Lambda$ up to cyclic permutation.

The notions of a string and a band can be defined for any zero-relation algebra, i.e., a bound quiver algebra $\Lambda=\K Q/I$ for which $I$ is generated by paths. Motivated by variants of the zeta function, we introduce a complex-valued \emph{zeta function for a zero-relation algebra} $\Lambda$, defined as $\zeta_\Lambda(t) := \prod\limits_{\B \in \Ba(\Lambda) } (1 - t^{|\B|})^{-1}$ for small $|t|$, where $|\B|$ is the length of $\B$. In the language of zeta functions, bands correspond to primes, and we define the \emph{band counting function} $\pi_\Lambda(n):= |\{ \B \in \mathrm{Ba}(\Lambda) : |\B| = n \}|$ analogous to the graph prime counting function \cite[Definition~2.11]{Terras_2010}.

To a zero-relation algebra $\Lambda$ whose defining relations have maximal length $N+1$, we associate its \emph{state graph} $G_\Lambda$, which is a directed graph whose vertices are the strings of length $N$, and whose arrows are the strings of length $N+1$. This construction is akin to the construction of the graph representation of a shift of finite type in symbolic dynamics \cite[\S~2.3]{symbolic_dynamics}. Another similar construction of graphs with certain words appeared in the study of the growth of Hilbert series of associative algebras by Ufnarovskii \cite[\S~2]{Ufnarovski1991Graphs}. Cyclic permutations of powers of bands for $\Lambda$ are in a bijective correspondence with cycles in $G_\Lambda$ (see Theorem \ref{thm:correspondence}). The Ihara zeta function for undirected graphs \cite{Sunada} and its generalisation to directed graphs \cite{KotaniSunada} (also see \cite{Horton2007IharaDigraphs,Terras_2010}) are also defined as products indexed by cycles. Inspired by the fact that the Ihara zeta function can be viewed as a special case of the Ruelle zeta function \cite[Equations~(4.5)--(4.7)]{Terras_2010}, we show that the zeta function $\zeta_\Lambda(t)$ admits a compact formula in terms of the adjacency matrix $A_\Lambda$ of $G_\Lambda$ (Theorem~\ref{thm: compact-form_of_zeta}).

The main contributions of the paper are based on the beautiful interplay between the zeta function and the $\mu$-series of a string algebra; these results also extend to special biserial algebras (see \S~\ref{sec 7}). Analogous to the the graph prime number theorem \cite[Theorem~10.1]{Terras_2010}, we prove a \emph{prime number theorem} for string algebras as stated below. (Say $f\sim g$ for functions $f,g:\mathbb{N}\to\mathbb{C}$ if $\lim\limits_{m\to\infty} \frac{f(m)}{g(m)}=1$.)

\begin{introthm}\label{thm:sapnt}
For a string algebra $\Lambda$, there are positive integers $C_\Lambda, L_\Lambda$ such that $$\pi_\Lambda(mL_\Lambda)\sim C_\Lambda \frac{R(A_\Lambda)^{mL_\Lambda}}{mL_\Lambda},$$ where $R(A_\Lambda)$ is the spectral radius of $A_\Lambda$.  
\end{introthm}

Given a non-domestic string algebra $\Lambda$, there are distinct bands $\B_1,\B_2$ such that both $\B_1'\B_2'$ and $\B_2'\B_1'$ are strings for some cyclic permutations $\B_i'$ of $\B_i$. This observation hints at $\Lambda$ being of exponential growth--indeed, we obtain this statement as a consequence of the above theorem.

\begin{introcor}\label{cor: expgr}
A non-domestic string algebra is of exponential growth.
\end{introcor}

Our final main result connects the domesticity of a string algebra with the rationality of its $\mu$-series.

\begin{introthm}\label{main}
A string algebra $\Lambda$ is domestic if and only if $\bar{\mu}_\Lambda(t)$ is rational.  
\end{introthm}

The rest of the paper is organised as follows. In \S~\ref{sec: 2}, we recall some results from spectral graph theory and number theory. In the context of zero-relation algebras, we define strings and bands in \S~\ref{sec: 3}, and then describe the construction of their state graphs followed by the properties of their zeta functions in \S~\ref{sec: 4}. Focusing on string algebras, we prove Theorem \ref{thm:sapnt} and Corollary \ref{cor: expgr} in \S~\ref{sec: 5} while the proof of Theorem \ref{main} occupies \S~\ref{sec: 6}. Finally, we discuss the extension of our results from string algebras to special biserial algebras in \S~\ref{sec 7}.

\subsection*{Acknowledgement} AK thanks Hemant Bhate for a lecture series as well as discussions on Ihara zeta functions.

\section{Preliminaries}\label{sec: 2}

Recall that a \emph{directed graph} is a pair $G=(V,E)$, where $V$ is a finite set, called the vertex set, and $E$ is an irreflexive binary relation on $V$, called the arrow set. Given an indexing $v_1,v_2,\cdots,v_n$ of $V$, the \emph{adjacency matrix} of $G$ is an $(n\times n)$-matrix $A_G=(a_{ij})_{i,j=1}^{n}$ where $a_{ij}$ is the number of arrows from $v_i$ to $v_j$ in $G$. A \emph{directed acyclic graph} (DAG) is a directed graph with no directed cycles. A \emph{strongly connected component} of $G$ is a maximal induced subgraph of $G$ in which there is a directed path between any two vertices.

\begin{rem}
Throughout the paper, the term \emph{subgraph} includes non-induced subgraphs.
\end{rem}

\begin{rem}
The results in this section are also valid when directed graphs are replaced by quivers, i.e., when loops and parallel arrows are allowed.
\end{rem}

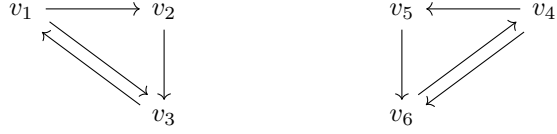
\begin{figure}[h]
\begin{center}
    \begin{tikzcd}[row sep=large, column sep=large]
        v_1 & v_2 & & v_5 & v_4 \\
            & v_3 & & v_6 &
        \arrow[from=1-1, to=1-2]
        \arrow[from=1-2, to=2-2]
        \arrow[shift left=1, from=1-1, to=2-2]
        \arrow[shift left=1, from=2-2, to=1-1]
        \arrow[from=1-5, to=1-4]
        \arrow[from=1-4, to=2-4]
        \arrow[shift left=1, from=1-5, to=2-4]
        \arrow[shift left=1, from=2-4, to=1-5]
    \end{tikzcd}
\end{center}
\caption{An example of a directed graph}
    \label{fig: directed graph ex}
\end{figure}

\begin{exmp}\label{ex: prelim, running 1}
If $G$ is the directed graph in Figure \ref{fig: directed graph ex}, then $A_G=$
$\begin{pmatrix}
M & 0\\
0 & M
\end{pmatrix}$,
where $M=\begin{pmatrix}
    0 & 1 & 1\\
    0 & 0 & 1\\
    1 & 0 & 0
\end{pmatrix}$. Note that $G$ has two strongly connected components, with vertex sets $\{v_1,v_2,v_3\}$ and $\{v_4,v_5,v_6\}$, each with adjacency matrix $M$. We have $\det(I-tA_G)=1-2t^2-2t^3+t^4+2t^5+t^6$.
\end{exmp}



\begin{rem}
For an $(n\times n)$-matrix $A$ with characteristic polynomial $\chi_A$, we have $\det(I-tA)=t^n\chi_A(t^{-1})$.
\end{rem}

For a matrix $A$, let $\mathrm{Spec}\, A$ denote its spectrum (multiset of eigenvalues with multiplicity) and $R(A)$ denote its spectral radius (maximum absolute value of an eigenvalue of $A$).

\begin{defn}
An $(n\times n)$ real matrix $A$ is said to be \emph{irreducible} if there does not exist a permutation matrix $P$ such that $A=P^t \begin{pmatrix}
    B & C\\
    0 & D
\end{pmatrix}P$, where $P^t$ denotes the transpose of $P$ and $B$ is an $r\times r$ matrix with $1\le r<n$.
\end{defn}

\begin{thm}[Perron and Frobenius]\label{thm: perron}\cite[Theorem~8.4.4, Corollary~8.4.6]{Horn_Johnson}\label{PF}
    Let $A$ be an irreducible $(n \times n)$ real matrix with non-negative entries and $n\geq2$. Then 
    \begin{enumerate}
        \item the spectral radius $R(A)$ is positive; and
        \item if $S(A)= \{ \lambda_1, \cdots, \lambda_k\}$ is the set of eigenvalues of $A$ having maximum modulus, then $$S(A) = \left\{ R(A) e^{\frac{2 \pi i a}{k}} : 1 \leq a \leq k \right\},$$ and each eigenvalue in $S(A)$ has algebraic multiplicity $1$.
    \end{enumerate}
\end{thm}

\begin{rem} \label{rem: scc-irr}
    If a directed graph is strongly connected then its adjacency matrix is irreducible.
\end{rem}

It is possible to compute the polynomial $\det(I-tA_G)$ using strongly connected components of $G$. We begin with a standard fact about DAGs.
\begin{rem}\label{rem: topo_order}
    There exists a topological ordering of the vertices of a DAG, i.e., there is an indexing $v_1,\cdots, v_n$ of vertices in a DAG such that whenever there is an arrow from vertex $v_i$ to vertex $v_j$, we have $i<j$.
\end{rem}
\begin{prop}\label{prop: det_strong_conn}
    Suppose $G$ is a directed graph with strongly connected components $G_1,\cdots, G_k$. Then $$\det(I-tA_G) = \prod_{i =1}^k \det(I-tA_{G_i}).$$
\end{prop}

\begin{proof}
For $1\leq i \leq k$, let $V_i$ be the vertex set of $G_i$. Define the \emph{condensation graph} $G'$ of $G$ with vertex set $\{G_i:1\le i\le k\}$, where an arrow exists from $G_i$ to $G_j$ if there is an arrow from a vertex in $V_i$ to a vertex in $V_j$ in $G$. By construction, $G'$ is a DAG. In view of Remark \ref{rem: topo_order}, we can relabel the vertices of $G'$ using a topological ordering, and hence there is a relabeling of the vertices of $G$ such that for any $i < j$, every vertex in $V_i$ has a smaller index than every vertex in $V_j$, while preserving the relative ordering of vertices within each $V_i$. Let $P$ be the permutation matrix associated with this relabeling, and let $A_G' = P^{-1}A_GP$ be the adjacency matrix of the relabeled graph. Since the determinant is invariant under similarity, we have
    \[
        \det(I - tA'_G) = \det(I - t P^{-1}A_G P) = \det(P^{-1}(I - tA_G)P) = \det(I - tA_G).
    \]
As a consequence of the topological ordering, $A_G'$ is an upper triangular block matrix with the adjacency matrices $A_{G_i}$ on the diagonal. This completes the proof.
\end{proof}

\begin{exmp}\label{exmp:using-SCC}
Continuing from Example \ref{ex: prelim, running 1}, note that $A_G$ is already an upper triangular block matrix. Moreover, both strongly connected components of $G$ are isomorphic to each other, and hence $\det(I-tA_G)=(\det(I-tM))^2=(1-t^2-t^3)^2$, which matches earlier computations, thus verifying Proposition \ref{prop: det_strong_conn}.
\end{exmp}

We end this section by recalling some standard results from number theory.

The \emph{M\"obius function} $\mu_{\textnormal{M\"{o}b}}:\mathbb{N}\to\mathbb{Z}$ is defined by
$ \mu_{\textnormal{M\"{o}b}}(m):=
\begin{cases}
1,&\text{if $\;m=1$}; \\
(-1)^r   ,&\text{if $\;m=\prod\limits_{i=1}^rp_i$ for distinct primes $p_i$};\\
0,&\text{otherwise}.
\end{cases}$

The \emph{Euler totient function} $\phi: \mathbb N \to \mathbb N$ is defined by $\phi(m)=|\{n\le m:n\text{ is relatively prime to }m\}|$. The following connections between these two functions are well-known.
\begin{rem}\label{rem: tot_prop}
    For $m \geq 1$, we have (1) $\frac{\phi(m)}{m} =\sum\limits_{d \mid m} \frac{\mu_{\textnormal{M\"{o}b}}(d)}{d}$ and (2) $\sum\limits_{d \mid m}\phi(d) = m$.
\end{rem}

\section{Strings and bands for zero-relation algebras}\label{sec: 3}
A \emph{quiver} is a directed graph possibly with parallel arrows as well as loops. More formally, a quiver $Q$ is a tuple $(Q_0,Q_1,s,e)$, where $Q_0$ is the set of vertices, $Q_1$ is the set of arrows, and $s,e:Q_1\to Q_0$ denote the source and end functions. A \emph{bound quiver} is a pair $(Q,I)$ such that $I$ is an admissible ideal of the path algebra $\K Q$, i.e., there exists some $N \geq 1$ such that $R_Q^2 \supseteq I \supseteq R_Q^{N+1}$, where the \emph{arrow ideal} $R_Q$ is the ideal generated by $Q_1$, and the corresponding \emph{bound quiver algebra} is the quotient algebra $\K Q/I$. Say that a bound quiver algebra $\K Q/I$ is a \emph{zero-relation algebra} if $I$ is generated by a set of paths. Let $\rho$ denote a (finite) set of paths that generate $I$ satisfying the property that no path in $\rho$ is a proper subpath of another path in $\rho$. An interested reader could refer to \cite{ASS} for literature on bound quiver algebras. The definitions stated below in the context of zero-relation algebras were originally defined in the context of string algebras by Butler and Ringel \cite{Butler1987AuslanderreitenSW}, but we use the notations from \cite{GKS}.

We use the lower case Roman letter $v$ to denote the vertices and $a,b$ to denote arrows of a quiver. Let us denote by $Q_1^{-1}$ the collection of formal inverses of the arrows in $Q_1$, each denoted by the corresponding upper case Roman letter. For $b\in Q_1$, set $s(B):=e(b)$ and $e(B):=s(b)$. The elements of $Q_1\sqcup Q_1^{-1}$ are called syllables--elements of $Q_1$ (resp. of $Q_1^{-1}$) are called \emph{direct} (resp. \emph{inverse}) syllables--and we denote them using Greek letters $\alpha,\beta,\gamma$ etc. Let $\rho^{-1}$ denote the set of formal inverses of the paths in $\rho$.

\begin{defn}
A \emph{string} of positive length for a zero-relation algebra is a finite sequence $\U=\alpha_n\dots\alpha_1$ of syllables with $n\geq1$, where for each $1\leq i<n$, we have $s(\alpha_{i+1})=e(\alpha_i)$ and $\alpha_{i+1}\ne \alpha_i^{-1}$, and no subsequence $\alpha_{i+j}\dots\alpha_i$ is in $\rho \cup \rho^{-1}$. Say that the \emph{length} of the string $\U$, denoted $|\U|$, is $n$.
\end{defn}

Given a zero-relation algebra $\Lambda$, let $\mathrm{Str}(\Lambda)$ denote the set of strings of positive length for $\Lambda$ and, for $k\geq1$, let $\mathrm{Str}_k(\Lambda)$ denote the set of strings of length $k$ for $\Lambda$. We suppress $\Lambda$ from the notation when the algebra $\Lambda$ is clear from the context.

The \emph{concatenation} of strings $\U=\alpha_n\dots\alpha_1$ and $\V=\beta_m\dots\beta_1$ is the word $\U\V:=\alpha_n\dots\alpha_1\beta_m\dots\beta_1$ provided it is a string. If the concatenation $\U\U$ exists, we write it as $\U^2$; more generally, $\U^n$ is defined similarly for $n\ge 2$ whenever it exists. The \emph{inverse} of $\U$ is $\U^{-1}:=\alpha_1^{-1}\dots\alpha_n^{-1}$. Note that $\U^{-1}$ is also a string. A positive length string is said to be a \emph{direct} (resp. \emph{inverse}) string if it consists only of direct (resp. inverse) syllables; otherwise it is said to be a \emph{mixed} string.

\begin{defn}
A string $\U=\alpha_n\dots\alpha_1$ is \emph{cyclic} if $e(\alpha_n)=s(\alpha_1)$. The string $\U$ is \emph{primitive} if it is cyclic and $\U\neq \V^k$ for every string $\V$ and every integer $k\ge 2$. The string $\U$ is a \emph{band} if it is primitive cyclic, $\alpha_n\in Q_1$, $\alpha_1\in Q_1^{-1}$, and $\U^k$ exists for all positive integers $k$. We use $\B$ to denote a band and $\mathrm{Ba}(\Lambda)$ to denote the collection of all bands for $\Lambda$ up to cyclic permutations of its syllables.
\end{defn}

\begin{rem}\label{rem:cycperm-iff}\cite[Remark~3.1.4]{GKS}
A cyclic string $\U=\alpha_n\dots\alpha_1$ is \emph{permutable} if $\alpha_j\dots\alpha_1\alpha_n\dots\alpha_{j+1}$ is a string for each $1\le j<n$. A mixed cyclic string is a cyclic permutation of a band power if and only if it is permutable.
\end{rem}

\begin{lem}\label{lem: cyc-perm}
If $\B\in \mathrm{Str}(\Lambda)$ is a band, then $\B$ is not a cyclic permutation of $\B^{-1}$.
\end{lem}
\begin{proof}
Let $\B=\alpha_n\dots\alpha_1$. Then $\B^{-1}=\alpha_1^{-1}\dots\alpha_n^{-1}$. Suppose, towards a contradiction, that $\B$ is a cyclic permutation of $\B^{-1}$. If $\B=\B^{-1}$, then by comparing corresponding syllables, we obtain $\alpha_n=\alpha_1^{-1}$, which implies that the cyclic permutation $\alpha_{n-1}\dots\alpha_1\alpha_n=\alpha_{n-1}\dots\alpha_1\alpha_1^{-1}$ of $\B$ is not a string, and hence that $\B$ is not permutable, a contradiction to Remark \ref{rem:cycperm-iff}.

Now suppose $\B=\alpha_{j+1}^{-1}\dots\alpha_n^{-1}\alpha_1^{-1}\dots \alpha_j^{-1}$, that is, $\B$ is a non-trivial cyclic permutation of $\B^{-1}$. Comparing corresponding syllables in the two expressions for $\B$, we obtain $\alpha_j\dots\alpha_1=\alpha_1^{-1}\dots\alpha_j^{-1}=\U$ (say). There are two cases.

If $|\U|=2m+1$ for some $m\ge 0$, then comparing corresponding syllables in the two expressions for $\U$, we obtain $\alpha_{m+1}=\alpha_{m+1}^{-1}$, a clear contradiction.

If $|\U|=2m$ for some $m\ge1$, then comparing corresponding syllables in the two expressions for $\U$, we obtain $\alpha_{m+1}\alpha_m=\alpha_m^{-1}\alpha_{m+1}^{-1}$, implying that $\alpha_{m+1}=\alpha_m^{-1}$. which is again a contradiction. This completes the proof.
\end{proof}

\section{Zeta Functions for zero-relation algebras}\label{sec: 4}

In this section, assume that $\Lambda=\K Q/I$ denotes a zero-relation algebra for a connected quiver $Q$, where $I=\langle\rho\rangle$ for a finite set $\rho$ of paths in $Q$. Let $N+1$ be the length of the longest path in $\rho$.

\begin{defn}\label{defn: ihzeta}
    The \emph{zeta function} of $\Lambda$ is the following function of the complex variable $t$, with $|t|$ sufficiently small:
    $$\zeta_\Lambda(t) := \prod_{\B \in \Ba(\Lambda) } (1 - t^{|\B|})^{-1}.$$
\end{defn}

\begin{rem}\label{rem: convergence_for_small_t}
    Because the zeta function has non-negative coefficients when expanded as a power series in the complex variable $t$, Landau's theorem \cite[p.~237]{apostol1976introduction} implies that both, the series and the product defining $\zeta_\Lambda(t)$, converge absolutely on an open disk of radius $R_\Lambda$ about origin with a singularity at $t=R_\Lambda$. In all computations involving the zeta function, we assume that $|t|<R_\Lambda$.
\end{rem}

We are interested in the coefficients of the power series expansion of the zeta function. A first step in that direction is to construct the \emph{state graph} associated with $\Lambda$.

\begin{defn}\label{defn: state_graph}
    The \emph{state graph} $G_\Lambda := G(\mathrm{Str}_N(\Lambda),\mathrm{Str}_{N+1}(\Lambda))$ of the algebra $\Lambda$ is a directed graph where, for $\U,\V\in\mathrm{Str}_N(\Lambda)$ and $\W\in\mathrm{Str}_{N+1}(\Lambda)$, there is an arrow from $\U$ to $\V$ labeled $\W$ if $\W=\alpha\U=\V\beta$ for some syllables $\alpha$ and $\beta$. For brevity, the adjacency matrix of $G_\Lambda$ will be denoted $A_\Lambda$.
\end{defn}

\begin{rem}
Continuing the notation from Definition \ref{defn: state_graph}, we can also refer to $\alpha$ as the label $\U\to\V$. 
\end{rem}

\begin{rem}
The strongly connected components of $G_\Lambda$ are in a bijective correspondence with  \emph{generalised meta-bands} studied in \cite[\S~5]{SKSK}.
\end{rem}

\begin{exmp}\label{exmp: GP23}
Consider the Gelfand-Ponomarev algebra $\mathrm{GP}_{2,3}$ presented by the data in Figure \ref{fig:quiver_l2}. The longest relation in $I$ is $b^3$ with length $3$,  $$\mathrm{Str}_2(\mathrm{GP}_{2,3}) = \{aB, Ba, Ab, bA, b^2, B^2 \} \text{ and }\mathrm{Str}_3(\mathrm{GP}_{2,3}) = \{ b^2 A, Ab^2, AbA, bAb, aBa, BaB, B^2 a, a B^2 \}.$$ Figure \ref{fig: gp23_state_graph} shows its state graph. 

\end{exmp}

\begin{figure}[ht!]
    \centering
    
    \begin{minipage}[c]{0.45\textwidth}
        \centering
        \[\begin{tikzcd}
            v
            \arrow["b"', from=1-1, to=1-1, loop, in=35, out=325, distance=10mm]
            \arrow["a", from=1-1, to=1-1, loop, in=145, out=215, distance=10mm]
        \end{tikzcd}\]
        \captionsetup{justification=centering} 
        \caption{$\mathrm{GP}_{2,3}$ with \protect\\ $\rho = \{a^2, b^3, ab, ba\}$}
        \label{fig:quiver_l2}
    \end{minipage}%
    \hfill
    \begin{minipage}[c]{0.5\textwidth}
        \centering
        \begin{tikzcd}[row sep=large, column sep=large]
            bA & b^2 & & B^2 & Ba \\
               & Ab  & & aB &
            \arrow[from=1-1, to=1-2, "bbA"]
            \arrow[from=1-2, to=2-2, "Abb"]
            \arrow[shift left=1, from=1-1, to=2-2, "AbA"]
            \arrow[shift left=1, from=2-2, to=1-1, "bAb"]
            \arrow[from=1-5, to=1-4, "BBa"']
            \arrow[from=1-4, to=2-4, "aBB"']
            \arrow[shift left=1, from=1-5, to=2-4, "aBa"]
            \arrow[shift left=1, from=2-4, to=1-5, "BaB"]
        \end{tikzcd}
        \caption{The state graph of $GP_{2,3}$}
        \label{fig: gp23_state_graph}
    \end{minipage}
    
\end{figure}
\begin{rem}
We briefly describe how the set of strings can be understood using the language of symbolic dynamics. For the alphabet $\mathcal{A}\coloneqq Q_1 \cup Q_1^{-1}$, $\mathrm{Str}({\Lambda}) \subseteq \mathcal{A}^{\mathbb Z}$ is a \emph{shift space} \cite[Definition~1.2.1]{symbolic_dynamics}. Since the length of the longest path that generates $I$ is $N+1$, all the strings in $\mathcal{A}^{\mathbb Z} \setminus \mathrm{Str}(\Lambda)$ contain a forbidden $(N+1)$-length substring. Thus, $\mathrm{Str}(\Lambda)$ is an \emph{$(N+1)$-step shift space of finite type} \cite[\S~2.1]{symbolic_dynamics}. Thanks to \cite[Theorem~2.3.2]{symbolic_dynamics}, it is straightforward to observe that $\mathrm{Str}(\Lambda)$ represents the \emph{edge shift space} \cite[Definition~2.2.5]{symbolic_dynamics} of the state graph $G_{\Lambda}$.
\end{rem}
\begin{defn}\label{defn: N_m}
Let $N_m$ be the number of strings in $\mathrm{Str}_m(\Lambda)$ that are cyclic permutations of band powers.
\end{defn}

\begin{thm}\label{thm:correspondence}
If $\mathrm{Tr}(A)$ denotes the trace of a matrix $A$, then $N_m=\mathrm{Tr}(A_\Lambda^m)$ for $m\ge 1$.
\end{thm}
\begin{proof}
If $A_\Lambda=(a_{ij})_{i,j=1}^n$, then $\mathrm{Tr}(A_\Lambda^m)=\sum\limits_{j_1=1}^n\dots \sum\limits_{j_m=1}^n a_{j_1j_2}a_{j_2j_3}\dots a_{j_mj_1}$, which from the definition of $G_\Lambda$ is equal to the number of cyclic walks of length $m$ in $G_\Lambda$. We demonstrate a bijection between cyclic permutations of band powers in $\mathrm{Str}_m(\Lambda)$ and cyclic walks of length $m$ in $G_\Lambda$.

Consider a cyclic walk $C=(\U_1\xrightarrow{\alpha_1}\U_2\xrightarrow{\alpha_2}\dots\xrightarrow{\alpha_{m-1}}\U_m\xrightarrow{\alpha_m}\U_1)$ in $G_\Lambda$. Clearly, the string $\U\coloneqq\alpha_m\dots\alpha_1$ is a cyclic string such that $\U^k$ is a string for all $k\ge 1$. We have $|\U^{N+1}|\ge N+1$. Since any direct or inverse string must have length $\le N$, we have that $\U^{N+1}$ is a mixed string. Thus, $\U$ is mixed and cyclic, and since it is also permutable, we have by Remark \ref{rem:cycperm-iff} that it is a cyclic permutation of a band power. Moreover, $\U_1$ is uniquely determined by $\U$ since $\U_1$ is the unique string in $\mathrm{Str}_N(\Lambda)$ satisfying $\U^{N+1}=\U_1\V$ for some string $\V$. Hence, the assignment $C\mapsto\U$ is bijective.
\end{proof}
\begin{cor}\label{cor:N_m}
For $m\geq1$, we have $N_m = \sum\limits_{\lambda \in \mathrm{Spec}\, A_\Lambda} \lambda^m$.
\end{cor}

We also obtain a linear recurrence relation for $N_m$ using the Cayley-Hamilton theorem. 
\begin{cor}\label{cor:linrecur}
Suppose the characteristic polynomial of $A_\Lambda$ is $\chi_{A_\Lambda}(t)=t^n-a_1t^{n-1}-a_2t^{n-2}-\cdots-a_n,$ then
$
N_{m+n}=a_1N_{m+n-1}+a_2N_{m+n-2}+\cdots+a_nN_m$ for each $m>0$.
\end{cor}

\begin{rem}\label{rem:minrecur}
    The minimal polynomial of $A_\Lambda$ can also be used in Corollary \ref{cor:linrecur} to obtain a potentially simpler recurrence relation.
\end{rem}



\begin{exmp}\label{exmp: GP23 cont.}
Continuing from Example \ref{exmp: GP23}, let $G$ and $A_G$ denote the state graph and its adjacency matrix respectively associated with the algebra $GP_{2,3}$. The minimal polynomial of $A_G$ is $t^3 - t - 1$. Remark \ref{rem:minrecur} yields the recurrence relation $N_{m+3} = N_{m+1} + N_{m}$ for $m\geq1$, and elementary computations give $N_1 = 0, N_2 = 4$,  and $N_3 = 6$.

From another perspective, for $m>3$, all the strings counted towards $N_m$ and ending with $a$ must have either $aB$ or $aB^2$ as a suffix. Similarly, all the strings counted toward $N_m$ and ending with a $B$ must have either $Ba$ or $B^2a$ as a suffix. This leads to $N_m = N_{m-2} + N_{m-3}$ for $m > 3$. 
\end{exmp}

The following is a well-known fact from linear algebra. 
\begin{fact}\label{fact: Tr-exp}
    For any matrix $A$, we have $\exp \mathrm{Tr}(A) = \det \exp (A)$.
\end{fact}

Now we are ready to state and prove an alternate expression for the zeta function.

\begin{thm}\label{thm: compact-form_of_zeta}
    For $|t|$ sufficiently small, we have 
    \begin{enumerate}
        \item $\log \zeta_\Lambda(t) = \sum \limits_{m\geq1} \frac{N_m}{m} t^m;$
        \item $t \frac{d}{dt}\log \zeta_\Lambda(t) = \sum\limits_{m\geq1}N_m t^m$;
        \item $\zeta_\Lambda(t) = \det (I - t A_\Lambda)^{-1}$.
    \end{enumerate}
\end{thm}

\begin{proof}
    (1)
    The proof is identical to the undirected graph case, as in \cite[Eq.~(4.5)]{Terras_2010}, but we include it here for completeness. Taking the logarithm of the zeta function given in Definition \ref{defn: ihzeta}, we get
    \begin{align*}
    \log \zeta_\Lambda(t) &= \log \left( \prod_{\B \in \Ba(\Lambda)} \left(1 - t^{|\B|}\right)^{-1} \right)= - \sum_{\B \in Ba(\Lambda)} \log \left(1 - t^{|\B|}\right) = \sum\limits_{\B \in \Ba(\Lambda)} \sum\limits_{k \ge 1} \frac{t^{k|\B|}}{k}, \intertext{where the second equality follows from absolute convergence of the series (see Remark \ref{rem: convergence_for_small_t}) while the third follows from Taylor series expansion. We rearrange the summation by grouping terms with the same total length $m = k|\B|$. For a fixed $m$, the inner sum runs over all bands $\B$ whose length $|\B|$ divides $m$. Substituting $k = m/|\B|$, we get}
    \log \zeta_\Lambda(t) &= \sum_{m \ge 1} \sum_{\substack{\B \in \Ba(\Lambda) \\ |\B| \text{ divides } m}} \frac{1}{m/|\B|} t^m = \sum_{m \ge 1} \frac{1}{m} \left( \sum_{\substack{\B \in \Ba(\Lambda) \\ |\B| \text{ divides } m}} |\B| \right) t^m.
    \end{align*}
    By Definition \ref{defn: N_m}, $N_m$ counts the number of cyclic permutations of strings of the form $\B^k$ where $m=|\B|k$. Since $\B$ is a band, the string $\B^k$ has exactly $|\B|$ distinct cyclic permutations. Thus, 
    $\sum\limits_{\substack{\B \in \Ba(\Lambda) \\ |\B| \text{ divides } m}} |\B| = N_m$. Substituting this in the above line yields the desired result.

    (2) It follows immediately from (1) above.

    (3) In view of (1) above, this proof is similar to that of \cite[Theorem~6.4.6]{symbolic_dynamics}. Indeed, Theorem \ref{thm:correspondence} and (1) together yield $$\log \zeta_\Lambda(t) = \sum_{m \geq 1} \mathrm{Tr} (A_\Lambda^m) \frac{t^m}{m} = \mathrm{Tr} \left( \sum_{m \geq 1} A_\Lambda^m \frac{t^m}{m} \right) = \mathrm{Tr} \left( \log(I - t A_\Lambda)^{-1} \right) = \log \det (I-tA_\Lambda)^{-1}, $$ where the second equality follows from the absolute convergence of the series (see Remark \ref{rem: convergence_for_small_t}), the third equality follows from Taylor series expansion, and the last equality follows from Fact \ref{fact: Tr-exp}. 
\end{proof}

\begin{rem}
Theorem \ref{thm: compact-form_of_zeta}(3) yields that the radius of convergence $R_\Lambda$ of $\zeta_\Lambda(t)$ satisfies $R_\Lambda= R(A_\Lambda)^{-1}$.
\end{rem}


Comparing Definition \ref{defn: ihzeta} with the Ihara zeta function \cite[Definition~2.11]{Terras_2010}, bands play the role of primes.
\begin{defn}
    For a zero-relation algebra $\Lambda$, the \emph{band counting function} $\pi_{\Lambda}:\mathbb N \to \mathbb N$ is defined as $$\pi_\Lambda(n):= |\{ \B \in \mathrm{Ba}(\Lambda) : |\B| = n \}|.$$
\end{defn}

We get the following (approximate) identities for the band counting function analogous to the graph prime number theorem \cite[Theorem 10.1]{Terras_2010} for the zero-relation algebras with similar proof using Theorem \ref{thm: compact-form_of_zeta}. 
\begin{cor}\label{cor: gtpnt}
For $m\geq1$, we have
\[
\makebox[\textwidth][l]{%
\begin{tabular}{@{}l@{\qquad\qquad\qquad\qquad}l@{}}
\text{(1)}\quad $\pi_\Lambda(m) = \frac{1}{m} \sum\limits_{d \mid m} \mu_{\textnormal{M\"{o}b}}\!\left( \frac{m}{d} \right) N_d$;
&
\text{(2)}\quad $\pi_\Lambda(m) \sim \frac{1}{m}\sum \limits_{\substack{\lambda \in \mathrm{Spec} A_\Lambda \\ |\lambda| = R(A_\Lambda)}} \lambda^m.$
\end{tabular}%
}
\]
\end{cor}

\begin{rem}\label{rem:spec_union}
Let $G_1, \cdots, G_k$ denote the strongly connected components of $G_\Lambda$ and let $A_i$ be the adjacency matrix of $G_i$. Remark \ref{rem: scc-irr} yields irreducibility of each $A_i$. Since $\mathrm{Spec} \, A_\Lambda=\bigsqcup\limits_{i=1}^k \mathrm{Spec} \, A_i$, we have \begin{equation} \label{eq 1}
    \sum_{\substack{\lambda \in \mathrm{Spec} A_\Lambda \\ |\lambda| = R(A_\Lambda)}} \lambda^m = \sum_{i=1}^k \sum_{\substack{\lambda \in S(A_i)\\R(A_i)=R(A_\Lambda)}} \lambda^m.
\end{equation} 

Since each $A_i$ is irreducible, Theorem \ref{PF}(2) implies that for each $m\ge 1$,
\[
\sum_{\lambda \in S(A_i)} \lambda^m =
\begin{cases}
|S(A_i)|\,R(A_i)^m, & \text{if } |S(A_i)| \text{ divides } m,\\
0, & \text{otherwise.}
\end{cases}
\]

Substituting this into Equation \eqref{eq 1}, Corollary \ref{cor: gtpnt}(2) yields 
\begin{equation} \label{eq 2}
    \pi_\Lambda(m)\sim\frac1m\sum_{\substack{\lambda \in \mathrm{Spec} A_\Lambda \\ |\lambda| = R(A_\Lambda)}} \lambda^m = \frac1m\sum \limits_{\substack{1\le i\le k \\ |S(A_i)|\text{ divides } m\\R(A_i)=R(A_\Lambda)}} |S(A_i)|R(A_i)^m.
\end{equation}
\end{rem}

\section{Prime number theorem for string algebras}\label{sec: 5}
In this section, we define string algebras, and then prove an analogue of the graph prime number theorem for them (Theorem \ref{thm:sapnt}), as a corollary of which it follows that non-domestic string algebras have exponential growth (Corollary \ref{cor: expgr}).
\begin{defn}\label{defn: stringalg}
    A zero-relation algebra $\Lambda = \K Q/I$ is a \emph{string algebra} if the following conditions hold.
    \begin{enumerate}
        \item For all $v \in Q_0$, there are at most two arrows $a, b \in Q_1$ such that $s(a) = s(b) = x$.
        \item For all $v \in Q_0$, there are at most two arrows $c, d \in Q_1$ such that $e(c) = e(d) = x$.
        \item For all $b \in Q_1$, there is at most one arrow $a \in Q_1$ such that $ab \notin I$.
        \item For all $b \in Q_1$, there is at most one arrow $c \in Q_1$ such that $bc \notin I$.
    \end{enumerate}
\end{defn}

A complete combinatorial classification of the finite-dimensional indecomposable modules of a string algebra $\Lambda$ was given by Butler and Ringel \cite{Butler1987AuslanderreitenSW}--we recall this classification result briefly. There is a \emph{string module} $\mathrm{M}(\U)$ associated with each string $\U \in \mathrm{Str} (\Lambda)$, and a \emph{band module} $\mathrm{B}(\B, n, \theta)$ associated with each tuple $(\B, n, \theta)$, where $\B \in \mathrm{Ba}(\Lambda)$, $n \in \mathbb{N}$ and $\theta \in \mathcal{K}^*$. Every finite-dimensional indecomposable $\Lambda$-module is isomorphic to either a string module or a band module. Two string modules $\mathrm{M}(\U)$ and $\mathrm{M}(\U')$ are isomorphic if and only if $\U'\in\{\U,\U^{-1}\}$. Similarly, two band modules $\mathrm{B}(\B, n, \theta)$ and $\mathrm{B}(\B', n', \theta')$ are isomorphic if and only if $\B'$ is a cyclic permutation of either $\B$ or $\B^{-1}$, $n = n'$ and $\theta = \theta'$. Furthermore, no string module is isomorphic to a band module.


\begin{prop}\label{prop: mu-pi}
    For a string algebra $\Lambda$, we have $\mu_{\Lambda}(m) = \frac{1}{2} \sum\limits_{d \mid m} \pi_\Lambda(d). $
\end{prop}
\begin{proof}
For a string algebra $\Lambda$, it follows from \cite{Butler1987AuslanderreitenSW} that $\mu_\Lambda(m)$ counts the number of band modules of total dimension $m$. Given a band $\B$, Lemma~\ref{lem: cyc-perm} implies that the classes of $\B$ and $\B^{-1}$ are two distinct elements of $\Ba(\Lambda)$; however, $\mathrm{B}(\B, n, \theta)\cong\mathrm{B}(\B^{-1}, n, \theta)$ for each $n\geq1$ and $\theta\in\mathcal K^*$. Since $\dim_\mathcal{K} \mathrm{B}(\B, n, \lambda) = n |\B|$, the proof is complete. 
\end{proof}

\begin{prop}\label{cor: domestic_for_leq1}
The following are equivalent for a string algebra $\Lambda$:
\begin{enumerate}
    \item $\Lambda$ is domestic;
    \item $\mathrm{Ba}(\Lambda)$ is finite; and
    \item $R(A_\Lambda) \leq 1$.
\end{enumerate}
\end{prop}

\begin{proof}
The equivalence $(1)\iff(2)$ is well-known (e.g., see \cite{Butler1987AuslanderreitenSW}), so we will only prove $(2)\iff(3)$.

Assuming (3), Equation \eqref{eq 2} gives $\lim\limits_{m\to\infty}\pi_\Lambda(m)=\lim\limits_{m\to\infty}\left|\frac1m\sum \limits_{\substack{\lambda \in \mathrm{Spec} A_\Lambda \\ |\lambda| = R(A_\Lambda)}} \lambda^m \right|\leq \lim\limits_{m\to\infty} \sum \limits_{i=1}^k\frac{|S(A_i)| R(A_i)^m}m=0$. Since $\pi_\Lambda(m)$ only takes non-negative integer values, we get $\pi_\Lambda(m)$ is eventually $0$, thus proving (2).

Conversely, if (2) holds, then $\pi_\Lambda(m)$ is eventually $0$. Hence, Equation \eqref{eq 2} gives $\lim\limits_{m\to\infty} \frac{|S(A_i)| R(A_i)^m}m=0$ for each $1 \leq i \leq k$. Hence, $R(A_i) \leq 1$ for each $1 \leq i \leq k$ which yields (3).
\end{proof}

    

\begin{prop}\label{prop:mu-N_m_relation}
    For a string algebra $\Lambda$, we have $\mu_{\Lambda}(m) = \frac{1}{2m} \sum\limits_{d \mid m} \phi \left( \frac{m}{d}\right) N_d.$
\end{prop}
\begin{proof}
    We have
    \[
      2\mu_{\Lambda}(m) = \sum_{n \mid m} \left( \frac{1}{n} \sum_{d \mid n} \mu_{\textnormal{M\"{o}b}} \left( \frac{n}{d} \right) N_d \right) = \sum_{d \mid m} \frac{N_d}{d} \sum_{r \mid \frac{m}{d}} \frac{\mu_{\textnormal{M\"{o}b}} (r)}{r} = \sum_{d \mid m} \left( \frac{\phi \left(\frac{m}{d} \right)}{\frac{m}{d}} \right) \frac{N_d}{d}= \frac{1}{m} \sum\limits_{d \mid m} \phi \left( \frac{m}{d}\right) N_d,
    \]
    where the first equality follows from Proposition \ref{prop: mu-pi} and Corollary \ref{cor: gtpnt} $(1)$, the second by interchanging the two summations and substituting $n = dr$, and the third using Remark \ref{rem: tot_prop} $(1)$.
\end{proof}

\begin{prop}\label{prop:mu_pi_asymptotic}
Suppose $R(A_\Lambda)>1$. Whenever $\pi_\Lambda(m) \neq 0$, we have $\mu_{\Lambda}(m)\sim \frac{1}{2}\pi_\Lambda(m)$.
\end{prop}
\begin{proof}
We write $R\coloneqq R(A_\Lambda)$ for brevity.
Proposition \ref{prop:mu-N_m_relation} and Corollary \ref{cor:N_m} together give
\begin{align*}
2m\mu_{\Lambda}(m)&=\sum\limits_{d\mid m}\phi \left(\frac{m}{d}\right)N_d=\sum\limits_{d\mid m}\phi\left(\frac{m}{d}\right)\left(\sum\limits_{\lambda\in\mathrm{Spec} \, A_\Lambda}\lambda^d\right)\\
&=\left(\sum_{\substack{d \mid m\\d<m}}\phi\left(\frac{m}{d}\right) \left(\sum\limits_{\lambda\in\mathrm{Spec}\, A_\Lambda}\lambda^d \right) \right)+\left(\sum_{\substack{\lambda\in\mathrm{Spec}\, A_\Lambda\\ |\lambda|<R}}\lambda^m \right)+\left(\sum_{\substack{\lambda\in\mathrm{Spec}\, A_\Lambda\\ |\lambda|=R}}\lambda^m \right)
\end{align*}
Writing the above sum as $S_1(m)+S_2(m)+S_3(m)$, notice that $S_1(m)\ge0$ thanks to Corollary \ref{cor:N_m}. We have
    \begin{align*}
     \frac{S_1(m)}{R^m}&\le\left(\sum_{\substack{d \mid m\\d<m}}\phi\left(\frac{m}{d}\right)\left(\frac{1}{R^{m-d}}\right) \left(\sum\limits_{\lambda\in\mathrm{Spec} \,A_\Lambda}\left(\frac{\lambda}{R} \right)^d \right) \right)
    \le \left(\sum_{\substack{d \mid m\\d<m}}\phi\left(\frac{m}{d}\right)\left(\frac{1}{R^{m/2}}\right) \left(\sum\limits_{\lambda\in\mathrm{Spec} \, A_\Lambda}\left(\frac{|\lambda|}{R} \right)^d \right) \right)\\
    &
    \le \left(\sum_{\substack{d \mid m\\d<m}}\phi\left(\frac{m}{d}\right)\left(\frac{1}{R^{m/2}}\right) \left(\sum\limits_{\lambda\in\mathrm{Spec}\, A_\Lambda}1 \right) \right)
    \le |\mathrm{Spec} \, A_\Lambda|\left(\sum_{\substack{d \mid m\\d<m}}\phi\left(\frac{m}{d}\right)\left(\frac{1}{R^{m/2}}\right) \right)\\
    &
    = |\mathrm{Spec} \, A_\Lambda|\frac{(m-1)}{R^{m/2}} \to 0 \text{ as } m\to\infty \text{ (since } R>1),
    \end{align*}
    where the second inequality follows from the fact that $m-d\ge \frac{m}{2}$ for any proper divisor $d$ of $m$ and the equality follows from \ref{rem: tot_prop}(2).
    
    We also have
    \[
    0\le\left|\frac{S_2(m)}{R^m}\right|\le\left|\sum_{\substack{\lambda\in\mathrm{Spec}\, A_\Lambda\\ |\lambda|<R}}\left(\frac{\lambda}{R}\right)^m \right| \le |\mathrm{Spec} \, A_\Lambda| \left(\frac{R'}{R}\right)^m,
    \]
     where $R'\coloneqq \max\{|\lambda|:\lambda\in\mathrm{Spec} \, A_\Lambda,\;|\lambda|<R\}$. Thus, $\lim\limits_{m\to\infty}\frac{S_2(m)}{R^m}=0$.

     If $\pi_\Lambda(m)\ne 0$, Equation \eqref{eq 2} yields $1\leq\frac{m\pi_\Lambda(m)}{R^m}\leq|\mathrm{Spec} \, A_\Lambda|$, and Corollary \ref{cor: gtpnt}(2) yields $S_3(m)\sim m\pi_\Lambda(m)$. Therefore, considering the sequence of $m_k$ where $\pi_\Lambda(m_k)\neq0$, we have $$\lim\limits_{k\to\infty}\frac{2\mu_{\Lambda}(m_k)}{\pi_\Lambda(m_k)}=\lim\limits_{k\to\infty}\frac{S_1(m_k)+S_2(m_k)+S_3(m_k)}{R^{m_k}}\frac{R^{m_k}}{\pi_\Lambda(m_k)}=1. \qedhere$$
\end{proof}

Now we are ready to prove Theorem \ref{thm:sapnt} and Corollary \ref{cor: expgr}.
\begin{proof}[Proof of Theorem~\ref{thm:sapnt}]
Let $L_\Lambda:= \mathrm{lcm}\{|S(A_i)|:1\le i\le k\}$ and $C_\Lambda:=\sum \limits_{\substack{1\le i\le k\\R(A_i)=R(A_\Lambda)}} |S(A_i)|$. Then Equation \eqref{eq 2} yields $$\pi_\Lambda(mL)\sim\frac{1}{mL}\sum \limits_{\substack{1\le i\le k\\R(A_i)=R(A_\Lambda)}} |S(A_i)|R(A_\Lambda)^{mL}=C_\Lambda\frac{R(A_\Lambda)^{mL}}{mL}.\qquad\qedhere$$ 
\end{proof}

\begin{proof}[Proof of Corollary~\ref{cor: expgr}]
For a non-domestic string algebra $\Lambda$, Proposition \ref{cor: domestic_for_leq1} yields $R(A_\Lambda)>1$. Thanks to Proposition \ref{prop:mu_pi_asymptotic}, it suffices to prove the exponential growth of $\pi_\Lambda$. Using the above theorem, for sufficiently large $m_0$, there exists $1<R'<R$ such that $\pi_\Lambda(mL)>(R')^{mL}$ for $m\ge m_0$. Let $n\ge 1$. There exists $m_1$ such that for all $m\ge m_1$, we have $m>\frac{n\ln(m)+n\ln(L)}{L\ln(R')}$ and hence $(R')^{mL}>(mL)^n$. Therefore, for $m\ge \max\{m_0,m_1\}$, we have $\pi_\Lambda(mL)>(R')^{mL}>(mL)^n$, thereby proving the exponential growth of $\pi_\Lambda$.
\end{proof}

\begin{exmp}
Continuing from Example \ref{exmp: GP23 cont.}, we have $R(A_G) > 1$ which implies that $N_m$ grows exponentially with $m$. Proposition \ref{prop:mu-N_m_relation} shows that $\mu(m)\ge \frac{N_m}{2m}$, which easily yields the exponential growth of $\mu(m)$.
\end{exmp}





\section{Proof of Theorem \ref{main}} \label{sec: 6}
Let $\Lambda$ be a string algebra. In this section, we connect domesticity of $\Lambda$ with the rationality of its $\mu$-series. We suppress $\Lambda$ from notations $\mu_\Lambda$, $\bar{\mu}_\Lambda$, and $\pi_\Lambda$ for brevity.

The forward direction of Theorem \ref{main} is the following result.
\begin{prop}\label{prop: rational}
    If a string algebra $\Lambda$ is domestic, then $\bar{\mu}(t)
    = \frac{1}{2}\sum\limits_{\B\in\mathrm{Ba}(\Lambda)}\frac{t^{|\B|}}{1 - t^{|\B|}}$ is a rational function.
\end{prop}

\begin{proof}
    By Proposition~\ref{cor: domestic_for_leq1}(2), $\mathrm{Ba}(\Lambda)$ is finite. Hence $\pi$ is eventually zero and has finite support, which we denote by $\{m_1,\dots,m_k\}$. Therefore
    \begin{align*}
    2\bar{\mu}(t)
    = 2\sum_{m\ge 1}\mu(m)t^m
    = \sum_{m\ge 1}\sum_{d\mid m}\pi(d)t^m
    = \sum_{i=1}^k \pi(m_i)\sum_{m \geq 1} t^{m_i m}
    = \sum_{i=1}^k \pi(m_i) \frac{t^{m_i}}{1 - t^{m_i}}
    = \sum\limits_{\B\in\mathrm{Ba}(\Lambda)}\frac{t^{|\B|}}{1 - t^{|\B|}},
    \end{align*}
    where the second equality follows from Proposition \ref{prop: mu-pi}. Since $\bar{\mu}(t)$ is a finite sum of rational functions, it follows that $\bar{\mu}(t)$ is itself rational.
\end{proof}



The rest of the section is devoted to the proof of the backward implication of Theorem~\ref{main}.

Suppose $\Lambda$ is non-domestic. Then  $R:=R(A_\Lambda)>1$ by Proposition \ref{cor: domestic_for_leq1}. The proof begins with a technique similar to that of Proposition \ref{prop:mu_pi_asymptotic} where we write $\mu(m)$ as a sum.

Using Proposition \ref{prop:mu-N_m_relation} and Corollary \ref{cor:N_m}, for each $m \in \mathbb N$, we have \begin{align*}
        \mu(m) = \frac{1}{2m} \sum_{d \mid m} \phi \left( \frac{m}{d} \right) N_d
               = \frac{N_m}{2m} + \frac{1}{2m} \sum_{\substack{d \mid m \\ d < m}} \phi \left( \frac{m}{d}\right)N_d 
               = \frac{1}{2m} \sum_{\lambda \in \mathrm{Spec} \, A_\Lambda} \lambda^m + \frac{1}{2m} \sum_{\substack{d \mid m \\ d < m}} \phi \left( \frac{m}{d}\right)N_d. 
    \end{align*}
    Writing the sum on the right hand side of the above line as $\mu(m) = \mu_1(m) + \mu_2(m)$ for each $m \in \mathbb N$, we also have $\bar{\mu}(t) = \bar{\mu}_{1}(t) + \bar{\mu}_{2} (t),$ where $\bar{\mu}_{i}(t) = \sum\limits_{m \geq 1} \mu_{i}(m) t^m$ for $ 1 \leq i \leq 2$. 

         Thanks to Remark \ref{rem:spec_union}, we know that the algebraic multiplicity of $R$, say $l$, as an eigenvalue of $A_\Lambda$ satisfies $l\ge1$. Hence we obtain $$\overline\mu_1(t)= \sum_{m\ge1}\frac{1}{2m} \left( l R^m +  \sum_{ \substack{\lambda \in \mathrm{Spec} \, A_\Lambda \\ \lambda \neq R}} \lambda^m\right)t^m= -\frac{1}{2} \left( l \log (1-Rt) + \sum_{ \substack{\lambda \in \mathrm{Spec} \, A_\Lambda \\ \lambda \neq R}} \log(1-\lambda t) \right).$$

    Denote the second summand in the parentheses of the last expression by $H(t)$. For each $\lambda\in \mathrm{Spec}\,A_\Lambda$ with $\lambda\neq R$, there exists $\varepsilon_\lambda>0$ such that (a branch of) $\log(1-\lambda t)$ is holomorphic on the disc $D(\frac1R,\varepsilon_\lambda)$. Since $\mathrm{Spec} \, A_\Lambda$ is finite, there is $\varepsilon'>0$ such that $H(t)$ is holomorphic on the disc $D(\frac1R,\varepsilon')$. Consequently, $\bar{\mu}_1(t)$ has a logarithmic singularity at $\frac1R$.

         Using Proposition \ref{prop:mu-N_m_relation} and Corollary \ref{cor:N_m}, we have
         \begin{align*}
             \mu_2(m) = \frac{1}{2m} \sum_{\substack{d \mid m \\ d < m}} \phi \left( \frac{m}{d}\right)N_d
             = \frac{1}{2m} \sum_{\substack{d \mid m \\ d < m}} \left( \sum_{\lambda \in \mathrm{Spec} \, A_\Lambda} \lambda^d \right) \phi \left( \frac{m}{d}\right)
             = \sum_{\lambda \in \mathrm{Spec} \, A_\Lambda} \left( \frac{1}{2m} \sum_{\substack{d \mid m \\ d < m}} \lambda^d \phi \left( \frac{m}{d}\right) \right).
         \end{align*}
         Thus, $$|\mu_2(m)| \le\sum_{\lambda \in \mathrm{Spec} \, A_\Lambda} \left( \frac{1}{2m} \sum_{\substack{d \mid m \\ d < m}} |\lambda|^d \phi \left( \frac{m}{d}\right) \right).$$
        
        \begin{align*}
        \intertext{\textbf{Case 1:} If $|\lambda|\le 1$, then $|\lambda|^m\le 1$ for all $m\ge 1$, and hence Remark \ref{rem: tot_prop}(2) gives}
        \frac{1}{2m} \sum_{\substack{d \mid m \\ d < m}} |\lambda|^d \phi\!\left(\frac{m}{d}\right)
        \le
        \frac{1}{2m} \sum_{\substack{d \mid m \\ d < m}} \phi\!\left(\frac{m}{d}\right) \notag
        =\frac{1}{2m}\left(\sum_{d\mid m}\phi\!\left(\frac{m}{d}\right)-\phi(1)\right) \notag
        =\frac{1}{2m}\,(m-1) \notag
        < \frac12.
        \end{align*}
        
        \begin{align*}
        \intertext{\textbf{Case 2:} If $|\lambda|>1$, then we use $\phi(\frac md)\le \frac md \le m$ and $d\le \frac m2$ for all proper divisors $d$ of $m$ to obtain}
        \frac{1}{2m} \sum_{\substack{d \mid m \\ d < m}} |\lambda|^d \phi\!\left(\frac{m}{d}\right)
        \le \frac{1}{2m} \sum_{\substack{d \mid m \\ d < m}} |\lambda|^d \, m \notag
        = \frac{1}{2}\sum_{\substack{d \mid m \\ d < m}} |\lambda|^d \notag
        \le \frac{1}{2}\sum_{d=1}^{\lfloor m/2\rfloor} |\lambda|^d \notag
        =\frac{1}{2} \left( \frac{|\lambda|^{\lfloor m/2\rfloor+1}-|\lambda|}{|\lambda|-1} \right).
        \end{align*}
        If $R_2$ is the radius of convergence of $\bar{\mu}_2(t)$, then by Cauchy-Hadamard formula, we have $$\frac{1}{R_2} = \limsup_{m \rightarrow \infty} |\mu_2(m)|^{\frac{1}{m}} \leq \limsup_{m \rightarrow \infty} \left| \frac{1}{2 } \left(\sum_{\substack{\lambda \in \mathrm{Spec} \, A_\Lambda \\ |\lambda| \leq 1 }} 1 + \sum_{\substack{\lambda \in \mathrm{Spec} \, A_\Lambda \\ |\lambda| > 1 }} \frac{|\lambda|^{\lfloor m/2\rfloor+1}-|\lambda|}{|\lambda|-1} \right) \right|^{\frac{1}{m}} = \sqrt{R}.$$ Hence, we conclude that $R_2 \geq \frac{1}{\sqrt{R}}> \frac 1R$. Thus, $\overline\mu_2(t)$ is holomorphic on a non-empty disc $D(\frac1R,\varepsilon'')$.
    
    
    If $\bar{\mu}(t)$ is rational, then it is meromorphic on $\mathbb{C}$. In particular, it is meromorphic on a non-empty disc $D(\frac1R,\varepsilon)$. Then the difference $\bar{\mu}_1(t)=\bar{\mu}(t)-\bar{\mu}_2(t)$ is also meromorphic on $D(\frac1R,\min \{\varepsilon, \varepsilon'' \})$, which contradicts the earlier conclusion that $\bar{\mu}_1(t)$ has a logarithmic singularity at $\frac1R$. Therefore, $\bar{\mu}(t)$ is not rational.

\section{Extension to special biserial algebras} \label{sec 7}

Special biserial algebras form a distinguished subclass of tame algebras that properly contains the class of string algebras--they are defined by dropping the requirement of being a zero-relation algebra in the definition of a string algebra. We recall some facts about the representation theory of special biserial algebras from \cite{WaschbuschSkowronski+1983+172+181} and \cite{WALD1985480}.

Consider a special biserial algebra $\Lambda = \mathcal{K}Q/I$. Let $\tilde{I}$ be generated by the set of paths $\U$ such that $(\U-\lambda \V) \in I$ for some path $\V$ sharing the source and end with $\U$ and $\lambda \in \mathcal{K}$ so that $\tilde{\Lambda}:= \mathcal{K}Q/ \tilde{I}$ is a string algebra. Every finite-dimensional indecomposable \(\Lambda\)-module is either a string module, a band module, or a projective-injective module. The canonical surjection $\Lambda \rightarrow \tilde{\Lambda}$ induces a bijection between the isomorphism classes of indecomposable \(\tilde{\Lambda}\)-modules and the isomorphism classes of indecomposable \(\Lambda\)-modules that are not projective-injective. Since there are only finitely many projective-injective indecomposables up to isomorphism, it follows that the $\mu$-series of $\Lambda$ and $\tilde\Lambda$ coincide. Therefore, the results of the above two sections extend to special biserial algebras.

\printbibliography

\end{document}